\documentclass{article}
\usepackage{amssymb}
\usepackage[utf8]{inputenc}
\usepackage{tikz-cd}
\usetikzlibrary{cd}
\usepackage{amsmath,amsfonts}
\usepackage{url}
\newcommand{\maps}{:}

\newcommand{\NN}{\mathbb{N}}
\newcommand{\bracket}[1]{\left \langle #1 \right\rangle}

\newcommand{\parens}[1]{\left( #1 \right)}
\newtheorem{lemma}{Lemma}[section]
\newtheorem{proposition}{Proposition}[section]
\newtheorem{theorem}{Theorem}
\newtheorem{remark}{Remark}[section]
\newtheorem{corollary}{Corollary}[section]
\newenvironment{proof}{\vspace*{2ex}\noindent {\em Proof:}
}{\hfill $\diamond$ \\[2ex]}

\title{Decompositions of Infinite-Dimensional $A_{\infty, \infty}$ Quiver Representations}
\author{Nathaniel Gallup and Stephen Sawin}
\date{\today}
\begin{document}

\maketitle

\begin{abstract}
Using linear algebraic methods we show that every (possibly infinite-dimensional) representation of a quiver with underlying graph $A_{\infty, \infty}$ is Krull-Schmidt, as long as the arrows in the quiver eventually point outward. 
\end{abstract}

\section{Introduction}

Most early techniques in the study of representations of quivers (e.g. Auslander-Reiten theory, the Euler form, etc.) assume that the quiver has finitely many vertices and that the vector spaces of these representations are finite-dimensional. In particular, Gabriel's Theorem \cite{gabriel1972} characterizes the quivers of finite-type as exactly those with underlying graph the ADE Dynkin diagrams. A nice survey of Gabriel's Theorem can be found in \cite{brion2008representations}, and we adopt most of Brion's notation in this paper. More recently there have been several results which weaken these finiteness hypotheses. 

In \cite{bautista-liu-paquette2011}, Bautista, Liu, and Paquette observed that any \textbf{locally finite-dimensional} representation (meaning each vector space is finite-dimensional) of various infinite quivers including an eventually outward  $A_{\infty, \infty}$ quiver is Krull-Schmidt (a direct sum of indecomposables) and identified the indecomposables in tht case as being in bijection with the connected full subquivers.  

In \cite{ringel2016}, Ringel was able to remove the finite-dimensional hypothesis from one direction of Gabriel's Theorem. There he proved that infinite-dimensional representations of (\textbf{finite}) ADE Dynkin quivers are Krull-Schmidt, and showed that these indecomposables are the same as those given in Gabriel's Theorem which correspond to the positive roots of the corresponding root system. 

In \cite{enochs2009injective}, Enochs, Estrada, and Rozas studied \textbf{injective} quiver representations over a general (not necessarily commutative) ring $R$. They proved that the indecomposable injective representations of certain infinite quivers (which include eventually outward $A_{\infty, \infty}$ quivers) called \emph{transfinite tree quivers} are in bijection with a subset of the connected full subquivers. Note that as they are working over a general ring, there is no locally finite-dimensional hypothesis. Here the term ``injective'' refers to injective objects in the category of representations and in fact (see Prop. 2.1 in \cite{enochs2009injective}) the linear maps in these representations are necessarily surjective.

In this paper, we simultaneously weaken the two finiteness conditions (those requiring finite quivers and locally finite-dimensional representations) as well as the injective representation requirement in the type-A case over a field. More specifically we show that every (possibly infinite-dimensional) representation of an eventually outward quiver with underlying graph $A_{\infty, \infty}$ is Krull-Schmidt. This eventually outward condition is in fact necessary: there exist representations of a not eventually outward $A_{\infty, \infty}$ quiver which are not Krull-Schmidt, an example of which is given at the end of the paper in Section \ref{sec: a non Krull-Schmidt representation}. We furthermore give a description of the indecomposables in this case, which are in bijection with the connected full subquivers which we can think of as ``extended simple roots''  which are limits of the classical simple roots. These match those from \cite{bautista-liu-paquette2011} and also those from \cite{enochs2009injective} if we restrict to injective representations. Our proof furthermore gives an algorithm for decomposing a given quiver representation into indecomposables.

In Sections \ref{sec: background}  we review the necessary quiver representation theory background. The remainder of the paper can be divided into two parts. 

In the first part we discuss representations of a general quiver $\Omega$ which is an eventually outward finitely branched tree, i.e. a union of finitely many journeys, the arrows of which eventually point away from its start. More specifically, in Section \ref{sec: the partial order of connected components} we define a certain partial order on the set $\mathcal{C}$ of connected full subquivers of $\Omega$ by taking the poset closure of two ``moves'' which we call \emph{reduction} and \emph{enhancement}. Using the eventually outward condition we prove that this is in fact a \textbf{well-founded} poset. 

In Section \ref{sec: poset filtrations of submodules} we then define a poset filtration of any representation $V$ of $\Omega$, i.e. an order-preserving function $F$ from $\mathcal{C}$ to the poset of subrepresentations of $V$. We will show that the successive quotients $F^\alpha / \sum_{\beta < \alpha} F^\beta$ of this poset filtration are isotypic, meaning a direct sum of indecomposables all of which are isomorphic.  

In Section \ref{sec: quotient and lift} we prove that for any connected full subquiver $\alpha \in \mathcal{C}$, the successive quotient $F^\alpha / \sum_{\beta < \alpha} F^\beta$ is supported on $\alpha$ and under mild conditions the morphism $F^\alpha \to F^\alpha / \sum_{\beta < \alpha} F^\beta$ lifts, i.e. that the subrepresentation $\sum_{\beta < \alpha} F^\beta$ is complemented, say by $W^\alpha$, inside of $F^\alpha$. Thus we obtain an \emph{almost gradation} $W$ of $F$, i.e. a function $W : \mathcal{C} \to \text{Sub}(V)$ such that for all $\alpha \in \mathcal{C}$ we have $F^\alpha = W^\alpha \oplus \sum_{\beta < \alpha} F^\beta$. We then prove that each $W_\alpha$ is an isotypic subrepresentation. Finally, using the fact that $\mathcal{C}$ is a well-founded poset, we use induction to prove that our poset filtration \emph{spans} $V$. This means that $V = \sum_{\alpha \in \mathcal{C}} W^\alpha$. 
 
In the second part of the paper, we specialize to the case where $\Omega$ is an eventually outward quiver of type $A_{\infty, \infty}$, and prove that any representation of such a quiver is Krull-Schmidt. We furthermore classify the indecomposables. Specifically, in Section \ref{sec: complete decomp of type A} we show that if $\Omega$ is eventually outward and of type $A_{\infty, \infty}$, then the partial order on the set $\mathcal{C}$ of connected subquivers of $\Omega$ is actually the product of two total orders on the set $\Omega_1$ of arrows of $\Omega$. It follows that, if $V$ is any representation of $\Omega$, then the poset filtration $F : \mathcal{C} \to \text{Sub}(V)$ is the intersection of two linear filtrations $L , R : \Omega_1 \to \text{Sub}(V)$. We then prove the general theorem that \textbf{any} almost gradation of a poset filtration which is the intersection of two linear filtrations is independent. All together this proves our main result that $V = \bigoplus_{\alpha \in \mathcal{C}} W_\alpha$ and thus every (possibly infinite dimensional) representation of an eventually outward type $A_{\infty, \infty}$ quiver $\Omega$ is Krull-Schmidt (a direct sum of indecomposables).

In Section \ref{sec: description of the indecomposables} we describe the isomorphism classes of the indecomposable subrepresentations of an eventually outward type $A_{\infty, \infty}$ quiver. As for any finite quiver of type $A_n$, they are in bijection with the connected subquivers, and are \emph{thin}, meaning one-dimensional at every vertex. Note that this means as long as the quiver is eventually outward, the indecomposables are independent of the directions of the arrows. This agrees with the indecomposables found in \cite{bautista-liu-paquette2011} for a locally finite-dimensional representation of an eventually outward type $A_{\infty, \infty}$ quiver, and with those found in \cite{enochs2009injective} for a (possibly infinite-dimensional) injective representation of an eventually outward type $A_{\infty, \infty}$ quiver.  Note the dimensions of these are $1$ at each vertex contained in the subquiver and $0$ otherwise.  The finite subquivers correspond to what are normally called positive roots, which are finite roots of Tits length $1$, or equivalently the image of simple roots under finitely many Weyl reflections.  The infinite quivers are naturally limits of these in the appropriate topology.

Finally in Section \ref{sec: a non Krull-Schmidt representation} we give an example of a representation of an $A_{\infty, \infty}$ quiver which is not Krull-Schmidt, i.e. is not a direct sum of indecomposable subrepresentations. This quiver is not eventually outward, as is required by the results of Section \ref{sec: complete decomp of type A}. 

\section{Background}\label{sec: background}

In this section we review the standard definitions in the theory of quivers and their representations and fix the notation that will be used throughout the paper.

\subsection{Quivers} 

A \emph{quiver}, denoted by $\Omega$, consists of a set $\Omega_0$ of \emph{vertices}, a set $\Omega_1$ of \emph{arrows}, a \emph{source function} $s: \Omega_1 \to \Omega_0$, and a \emph{target function} $t: \Omega_1 \to \Omega_0$. We say that the quiver has an  \emph{underlying graph} if every pair of vertices in $\Omega_0$ are the source and target of at most one arrow in $\Omega_1$ (in some order), and no vertex is the source and target of the same arrow.  Then the underlying graph has a vertex for each vertex of $\Omega$ and an undirected edge for each arrow of $\Omega$.  The quiver is a tree if the underlying graph is a tree. 

A vertex $x \in \Omega_0$ is a \emph{source} if it is not the target of any arrow, and a \emph{sink} if it is not the source of any arrow.

\subsection{Subquivers}

A \emph{subquiver} of a quiver $\Omega$ is a quiver $\Omega'$ such that $\Omega'_0 \subseteq \Omega_0$, $\Omega'_1 \subseteq \Omega_1$, $s' = s|_{\Omega'_1}$, and $t' = t|_{\Omega'_1}$. A \emph{full} subquiver is one that contains every arrow whose source and target are contained in it.

From now on, we will assume that the underlying graphs of all quivers in this paper are trees. Notice that means the removal of any arrow $e$ of a quiver divides the quiver into two connected subquivers, the one including $t(e)$ which we call the subquiver \emph{in front} of $e$ and the one containing $s(e)$ which we call the subquiver \emph{behind} $e$. In general, any vertex, arrow, or subquiver contained in the former will be said to be \emph{in front of} $e$ and $e$ will be said to \emph{point towards it}, while any in the latter will be said to be \emph{behind} $e$, and $e$ will be said to \emph{point away from it}.

\subsection{Walks and Paths} 

A (finite or infinite) \emph{walk} $w$ in a quiver with an underlying graph is a map from a connected (finite or infinite) interval of the integers to $\Omega_0$ so that if $i$ and $i+1$ are in the interval, there is an arrow whose source and target are $w_i$ and $w_{i+1}$. We then say $w$ \emph{contains} that arrow $e$.  We say $e$ is \emph{oriented consistently} if $s(e)=w_i$ and $t(e)=w_{i+1}$, and \emph{oriented inconsistently} if $s(e)=w_{i+1}$ and $t(e)=w_i$.  Adding an integer to every element of the interval results in a new map we will consider the same walk. If the interval has a minimal $i$ we define $s(w)=w_i$ and call the walk a \emph{journey}, if the interval has a maximal $j$ we define $t(e)=w_j$.  If $w$ has a target and $w'$ has a source and these two vertices are the same we can define the \emph{concatenation}  walk $w'w$ in the obvious way.  An \emph{injective walk} is one where the map is injective.

\subsection{Types of Quivers and Walks}\label{subsection: types of quivers and walks}

Since $\Omega$  is a tree then every two vertices are connected by a unique minimal walk. We say that $\Omega$ is a \emph{finitely-branched} tree if it is a union of finitely many injective paths, or equivalently for any vertex $x$ it is the union of finitely many injective journeys starting at $x$.  Finally, if every injective journey contains at most finitely many arrows oriented inconsistently, we say that $\Omega$ is \emph{eventually outward}.

These two types of quivers can be helpfully redefined in terms of a strict partial order $\prec$ on the set 
 $\Omega_1$ of arrows of $\Omega$ defined by $f \prec e$ if and only if $f$ is behind $e$ and $e$ is in front of $f$.

\begin{lemma} \label{lem: partial order on edges }
 $\Omega$ is eventually outward if and only if its $\prec$ is well-founded (no infinite descending chains).  If $\Omega$ is a tree it is finitely branched if and only if its $\prec$ has no infinite antichains (no two elements are related).
\end{lemma}

\begin{proof}
  If $\Omega$ contains an injective journey with infinitely many arrows inconsistently oriented, the set of inconsistently oriented arrows in that journey is an infinite descending chain.  If $\cdots \prec e_3 \prec e_2 \prec e_1$ is an infinite downward chain, there is a unique injective journey $p_i$ from the target of each $e_i$ to the source of each $e_{i+1}$, and these are all distinct because if $j>i$ then $p_j$ is behind $e_{i+1}$ and $p_i$ is in front.  Concatenating $\cdots p_2e_2^{-1}p_1e_1^{-1}$ gives a journey with each $e_i$ oriented inconsistently.

  If $\Omega$ is finitely branched, let $p_1, \ldots, p_n$ be a set of injective paths covering it.  If $E$ is an  antichain, there can be at most two  elements of $E$ contained in any given path, or else one would be in front of the other.  Thus $E$ is finite.  On the other hand if $\Omega$ is not finitely branched, then for some $x$ there is an infinite sequence $p_1, p_2, \ldots$ of injective journeys from x such that each journey $p_i$ contains a first arrow $e_i$ not in any of the previous journeys, and after that all the arrows are not in any other journeys.  Of those infinitely many arrows, either an infinite set of them point away from $x$ or an infinite set point towards $x$.  Each of these sets separately forms an antichain, because the path connecting $e_i$ to $e_j$ will follow the reverse of $p_i$ to the final point of shared intersection with $p_j$, then out $p_j$. Thus there is an infinite antichain.
\end{proof}

\begin{lemma} \label{lem: finite pointing towards}
  If $\alpha$ is a connected full nonempty subquiver of an eventually outward finitely branching tree, then there are only finitely many arrows that point towards $\alpha$.
\end{lemma}

\begin{proof}
  Being finitely branching, the quiver is a union of finitely many journeys from an $x \in \alpha$. Since $\alpha$ is connected all arrows that point towards $\alpha$ point towards $x$, therefore the result follows from eventually outward.
\end{proof}

\subsection{Quiver Representations}

A \emph{representation} $(V, f)$ of a quiver $Q$ over a field $\mathbb{F}$ consists of an $\mathbb{F}$-vector space $V_i$ at every vertex $i \in Q_0$ and a linear map $f_e : V_{s(e)} \to V_{t(e)}$ for every arrow $e \in Q_1$. If $(V, f)$ and $(W, g)$ are representations of $Q$, a \emph{morphism of representations} $T : V \to W$ consists of a linear map $T_i : V_i \to W_i$ for all vertices $i \in Q_0$ making the appropriate diagram commute.

Direct sums of representations of $Q$ are defined in the obvious way. A quiver representation is called \emph{indecomposable} if it is nonzero and is not a direct sum of two proper, non-zero subrepresentations, and is called \emph{Krull-Schmidt} if it is a direct sum of indecomposable subrepresentations.

\subsection{The Transport of a Subspace}

Let $\Omega$ be a quiver, $V$ a representation of $\Omega$, $w$ a finite injective walk  in $\Omega$, and $W \subset V_{s(w)}$ a subspace. We define the \emph{transport} of $W$ via $w$ to $t(w)$ recursively as follows: 

\begin{itemize}
    \item If $w$ contains a single arrow $e$ oriented consistently  then the transport of $W$ is $f_e[W] \subseteq V_{t(w)}$. 
    \item If $w$ contains a single arrow $e$ oriented inconsistently then the transport of $W$ is $f_e^{-1}[W] \subseteq V_{t(w)}$.
    \item If $w$ is the concatenation of path $w_1$ and $w_2$ then the transport of $w$ is the transport via $w_2$ of the transport via $w_1$ of $W$.
\end{itemize}

Notice that if $W' \subseteq W \subseteq V_{s(w)}$ then the transport of $W'$ is contained in the transport of $W$. 

\begin{lemma} \label{lm:transport}
Let $\Omega$ be a quiver which is a tree, $V$ a representation of $\Omega$, $i \in \Omega_0$ a vertex, and $W_i \subseteq V_i$ a subspace. For each vertex $j \in \Omega_0$, define $W_j$ to be the transport of $W_i$ to $j$ along the unique injective  walk starting at $i$ and ending at $j$. Then $W$ is the maximal subrepresentation of $V$ whose value at $i$ is $W_i$ and such that for each arrow $e$ pointing away from $i$, $f_e$ restricted to $W$ is onto. It is the minimal subrepresentation $W$ of $V$  whose value at $i$ is $W_i$ and such that for each arrow $e$ pointing towards  $i$, $f_e$ projected onto $V/W$ is one-to-one.
\end{lemma}

\begin{proof}  Elementary.

\end{proof}
\section{The Partial Order of Connected Components} \label{sec: the partial order of connected components}

Let $\Omega$ be an eventually outward finitely branched tree quiver.  Let $\mathcal{C}$ be the set of connected full subquivers of $\Omega$. If $\alpha \in \mathcal{C}$, the \emph{complement} of $\alpha$ in $\Omega$, denoted by $\alpha'$, is the full subquiver of $\Omega$ containing the vertices which are not in $\alpha$. Then $\alpha'$ is a disjoint union of finitely many components, with each component connected to $\alpha$ by a single arrow, which we call a \emph{boundary arrow} of $\alpha$. If $e$ connects one of the components of $\alpha'$ to $\alpha$, we denote this component by $\alpha'_e$. Let $\text{ba}(\alpha)$ denote the set of boundary arrows of $\alpha$. Thus we have the following: 
\begin{equation*}
    \alpha' = \bigsqcup_{e \in \text{ba}(\alpha)} \alpha'_{e}
\end{equation*}

Furthermore, we denote the subquiver of $\Omega$ which includes $\alpha'_e$ and $e$ and both the source and target of $e$ (one of which is already in $\alpha'_e$) by $\alpha''_e$.

If $e$ is any arrow not contained in $\alpha$ then $\alpha$ is contained entirely in one of the two components of $\Omega$ created by the removal of $e$, and therefore $e$ either points towards it or away from it.  Let the set of boundary arrows of $\alpha$  which point towards $\alpha$ be $\text{iba}(\alpha)$ and the set pointing away be $\text{oba}(\alpha)$. if it points from $\alpha$ to $\alpha'_e$. More generally, if $e$ is an arrow not in $\alpha$, then all of $\alpha$ is contained in one of the two subquivers into which the removal of $e$ separates $\Omega$, and therefore $\alpha$ is either in front of $e$ or behind it.

Define a relation on the set $\mathcal{C}$ of connected subquivers of $\Omega$ by letting $\alpha > \beta$ if one can get from $\alpha$ to $\beta$ by a sequence of the following moves:
\begin{enumerate}

    \item We say that $\beta$ is a \emph{reduction} of $\alpha$ if there exists an arrow $e$ in $\alpha$ such that $\beta$ is the portion of $\alpha$ behind $e$ (in terms of its orientation in $\Omega$). Thus $e$ is a boundary arrow of $\beta$ pointing away.
    
    \item We say that $\beta$ is an \emph{enhancement} of $\alpha$ if there exists an arrow $e$ in $\beta$ such that $\alpha$ is the portion of $\beta$ in front of $e$. Thus $e$ is a boundary arrow of $\alpha$ pointing towards it.

\end{enumerate}

\begin{proposition} \label{prop: no infinite sequence of reduction and enhancement}
There is no infinite sequence $\alpha_1, \alpha_2, ...$ such that each $\alpha_{n+1}$ is a nontrivial reduction or enhancement of $\alpha_n$. 
\end{proposition}

\begin{proof}
  Suppose for contradiction that such a sequence exists. For each $n \in \mathbb{N}$, let $S_n$ be the set of arrows $e$  not in $\alpha_n$ which point towards $\alpha_n$ and such that the journey from $e$ to $\alpha$ goes through a boundary arrow of $\alpha_n$ that also points towards $\alpha_n$. By Lem. \ref{lem: finite pointing towards}, $S_n$ is finite. Notice first that if $\alpha_{n+1}$ is a reduction or enhancement of $\alpha$ and $e \in S_{n+1}$, then $e \in S_n$, because in the first case the same journey works, and in the second the same journey works perhaps extended to go through the arrow of enhancement.  So $S_{n+1} \subseteq S_n$, and in the case of an enhancement $S_{n+1} \subset S_n$ because if we enhanced on boundary arrow $e$ of $\alpha_n$ then $e \in S_{n}-S_{n+1}$.  Therefore there can be only finitely many enhancements, and  there exists some $N \in \mathbb{N}$ such that for all $n \geq N$, $\alpha_{n + 1} < \alpha_n$ is a reduction. For each $n \geq N$ let $e_{n - N}$ be the arrow in $\alpha_n$ along which the reduction $\alpha_{n + 1} < \alpha_n$ occurs. Then we obtain a sequence of arrows $e_1, e_2 , \ldots$ such that for all $m > n$, $e_m$ is behind $e_n$.

Since $\Omega$ is finitely branched, it is the union of finitely many injective journeys, and each $e_n$ must be on one of these journeys, so infinitely many $e_{n_1}, e_{n_2}, \ldots$ lie on a single journey.  If any $e_{n_i}$ is oriented conistently with the journey, then only finitely many other $e_{n_j}$ can be behind it, which contradicts the fact that it was an infinite sequence.  But if all the $e_{n_i}$ are oriented inconsistently, this contradicts the fact that $\Omega$ was evenetually outward.

\end{proof}

\begin{corollary}
The relation $>$ on $\mathcal{C}$ defined above is a well-founded strict partial order.
\end{corollary}

\begin{proof}
Transitivity of the relation is trivially satisfied.    To show asymmetry, if $\alpha < \beta < \alpha$ then we obtain an infinite decreasing sequence of moves, contradicting Prop. \ref{prop: no infinite sequence of reduction and enhancement}. Irreflexivity follows from asymmetry.  Now that we know $<$ is a strict partial order, the aforementioned proposition also implies that it is well-founded as it shows there are no infinite strictly decreasing chains. 
\end{proof}

\section{Poset Filtrations of Subrepresentations} \label{sec: poset filtrations of submodules}

Let $\Omega$ be a quiver which is a tree, and $V$ a representation of $\Omega$. If $e$ is any arrow of $\Omega$, define two subrepresentations of $V$ as follows.  Define the subrepresentation $V^{e,+} \subset V$ by
\[V^{e,+}_i= \begin{cases} V_i & \text{ if } i \text{ is behind } e\\
\text{transport of } V_{t(e)} & \text{ else}.
\end{cases}
\]
  Define $V^{e,-} \subset V$ by
  \[V^{e,-}_i= \begin{cases} \{0\}_i \subset V_i & \text{ if } i \text{ is in front of } e\\
\text{transport of } \{0\}_{t(e)} & \text{ else}.
\end{cases}
\] 
It is easy to check that each of these is indeed a subrepresentation.

Define a \emph{virtual arrow} $E$ to be a path, i.e. an oriented injective journey with all arrows oriented away from the source, which is either infinite or ends in a leaf (thus it could consist of a single leaf vertex). If $\alpha \in \mathcal{C}$ is a full connected subquiver and $E$ lies entirely within $\alpha$ we say $E$ is a \emph{virtual boundary arrow} of $\alpha$, and denote by $\text{vba}(\alpha)$ the set of all virtual boundary arrows of $\alpha$.  If $E$ is a virtual arrow, define a subrepresentation $V^{E,+} \subset V$ by 
\[V^{E,+}_i= \begin{cases} V_i & \text{ if } i \in E\\
\text{transport of } V_{j} & \text{ else}
\end{cases}
\]
where $j$ is the closest vertex in $E$ to $i$, in the sense that the unique injective walk from $j$ to $i$ does not contain any other vertices of $E$. 

\begin{remark}\label{rmk: start of virtual boundary arrows does not matter}
 Notice that this is a subrepresentation, and also that if $E' \subset E$ is also a virtual arrow, then $V^{E,+}=V^{E',+}$, so this does not depend on where $E$ starts.
\end{remark}

If $\alpha \in \mathcal{C}$ is a connected subquiver and $e \in \text{ba}(\alpha)$, define $V^{e,\alpha}$ to be $V^{e,+}$ if $e \in \text{iba}(\alpha)$ and $V^{e,-}$ if $e \in \text{oba}(\alpha)$, and if $E \in \text{vba}(\alpha)$ define $V^{E,\alpha}=V^{E,+}$.  Finally, let $\text{b}(\alpha)= \text{ba}(\alpha) \cup \text{vba}(\alpha)$. 

Define a function $F : \mathcal{C} \to \text{Sub}(V)$, denoted $\alpha \mapsto F^\alpha$, by
\begin{equation} \label{eq:F definition}
  F^\alpha = \bigcap_{d \in \text{b}(\alpha)} V^{d,\alpha}.
  \end{equation}

In general a \emph{poset filtration} of an $R$ module $M$ consists of a partially ordered set $(P , \leq )$, and a function $F: P \to \text{Sub}(M)$ which is order-preserving, meaning $p \leq q$ implies $F^p \subseteq F^q$.

\begin{proposition}\label{prop: F is a poset filtration}
$F$ is a poset filtration of $V$, i.e. if $\beta \leq \alpha$ are connected subquivers then $F^\beta \subseteq F^\alpha$.
\end{proposition}

\begin{proof}

It suffices to check this if $\beta$ is obtained from $\alpha$ by a single reduction or enhancement.  First consider a reduction, with $e$ an arrow contained in $\alpha$ and $\beta$ the portion of $\alpha$ behind $e$.   Then $\beta$ and $\alpha$ share all the same boundary arrows except that $\beta$ has one outward pointing boundary arrow $e$ which $\alpha$ does not share and $\alpha$ has some number of boundary arrows which $\beta$ does not share, all of which are in front of $e$. Furthermore, every virtual boundary arrow of $\beta$ is also a virtual boundary arrow of $\alpha$ but $\alpha$ has some number of virtual boundary arrows which are not contained in $\beta$, and hence by Remark \ref{rmk: start of virtual boundary arrows does not matter} we can assume are entirely contained in $\beta_e'$. 

Therefore for any vertex $i$ in $\Omega$ we have the following:

\begin{equation}\label{eq: V alpha i intersection}
    F^\alpha_i = \left[ \bigcap_{d \in \text{b}(\alpha) \cap \text{b}(\beta)} V^{d , \alpha}_i \right] \cap \left[ \bigcap_{d \in \text{b}(\alpha) \smallsetminus \text{b}(\beta)} V^{d , \alpha}_i \right] 
\end{equation} 
\begin{equation}\label{eq: V beta i intersection}
     F^\beta_i = \left[ \bigcap_{d \in \text{b}(\alpha) \cap \text{b}(\beta)} V^{d , \beta}_i \right] \cap V^{e , \beta}_i. 
\end{equation}

First notice that because $\beta \subseteq \alpha$, for any $d \in \text{ba}(\alpha) \cap \text{ba}(\beta)$, $d$ is outward pointing with respect to $\alpha$ if and only if it is outward pointing with respect to $\beta$, and hence $V^{d, \alpha}_i = V^{d , \beta}_i$. Hence the first terms of the intersections in Equations \ref{eq: V alpha i intersection} and \ref{eq: V beta i intersection} are equal.  Now we consider two cases. 

Case 1: If $i$ is in front of $e$, then $V_i^{e , \beta} = \{ 0 \}$ hence $F_i^\beta = \{ 0 \}$ so trivially we have $F_i^\beta \subseteq F_i^\alpha$. 

Case 2: If $i$ is behind $e$, then $V_i^{e , \beta}$ is the transport of $\{ 0 \}_{t(e)}$ to $i$. But notice that for $d \in \text{b}(\alpha) \smallsetminus \text{b}(\beta)$, $V^{d, \alpha}_i$ is the transport of $V^{d, \alpha}_{t(e)}$ to $i$. Since $\{ 0 \}_{t(e)} \subseteq V^{d, \alpha}_{t(e)}$ and transport of subspaces preserves inclusion, it follows that $V_i^{e, \beta} \subseteq V^{d, \alpha}_i$ for all $d \in \text{b}(\alpha) \smallsetminus \text{b}(\beta)$, and therefore the second term Equation \ref{eq: V beta i intersection} is contained in the second term of Equation \ref{eq: V alpha i intersection}. 

Now consider an enhancement from $\alpha$ to $\beta$ along an inward boundary arrow $e$ of $\alpha$ contained in $\beta$.  Then $\alpha$ and $\beta$ share all the same boundary/ virtual boundary  arrows except that $\alpha$ has one arrow $e$ pointing towards it which $\beta$ does not share and $\beta$ has boundary/virtual boundary arrows all behind $e$ which $\alpha$ does not share. Therefore for a vertex $i$ in $\Omega$ we have the following:  
    
\begin{equation}\label{eq: V beta i intersection enhancement}
    F^\beta_i = \left[ \bigcap_{d \in \text{b}(\alpha) \cap \text{b}(\beta)} V^{d , \beta}_i \right] \cap \left[ \bigcap_{d \in \text{b}(\beta) \smallsetminus \text{b}(\alpha)} V^{d , \beta}_i \right] 
\end{equation} 
\begin{equation}\label{eq: V alpha i intersection enhancement}
     F^\alpha_i = \left[ \bigcap_{d \in \text{b}(\alpha) \cap \text{b}(\beta)} V^{d , \alpha}_i \right] \cap V^{e , \alpha}_i 
\end{equation}

As before, the first terms of the intersections in Equations \ref{eq: V beta i intersection enhancement} and \ref{eq: V alpha i intersection enhancement} are equal and the third terms in Equations \ref{eq: V beta i intersection enhancement} and \ref{eq: V alpha i intersection enhancement} are equal. Now we consider two cases.

Case 1: If $i$ is behind $e$, then $V_i^{e , \alpha} = V_i$, hence the second term of Equation \ref{eq: V beta i intersection enhancement} is trivially contained in the second term of Equation \ref{eq: V alpha i intersection enhancement}, yielding $F_i^\beta \subseteq F_i^\alpha$. 

Case 2: If $i$ is in front of $e$, then $V_i^{e , \beta}$ is the transport of $V_{t(e)}$ to $i$. But notice that for $d \in \text{b}(\beta) \smallsetminus \text{b}(\alpha)$, $V^{d, \beta}_i$ is the transport of $V^{d, \beta}_{t(e)}$ to $i$. Since $V^{d, \beta}_{t(e)} \subseteq V_{t(e)}$ and transport of subspaces preserves inclusion, it follows that $V_i^{e, \beta} \subseteq V^{d, \alpha}_i$ for all $d \in \text{b}(\alpha) \smallsetminus \text{b}(\beta)$, and therefore the second term in Equation \ref{eq: V beta i intersection enhancement} is contained in the second term of Equation \ref{eq: V alpha i intersection enhancement}. So $F_i^\beta \subseteq F_i^\alpha$.

\end{proof}

\begin{lemma} \label{lm:V-onto}
If $e$ is an arrow in $\alpha$ with exactly one literal or virtual boundary arrow of $\alpha$ behind it, then $f_e$ is onto when restricted to $F^\alpha$.  
\end{lemma}

\begin{proof} Let $b$ be the boundary behind $e$, and let $X=V^{b,\alpha}_{s(e)}$.  For each literal or virtual boundary $b_i$ of $\alpha$ other than $b$, let $Y_i=V^{b_i,\alpha}_{t(e)}$.  Then we have
  \[F^\alpha_{s(e)}= X \cap \bigcap_i f_e^{-1}[Y_i] \qquad F^\alpha_{t(e)}= f[X] \cap \bigcap_i Y_i.\]
  If $y \in f[X] \cap \bigcap_i Y_i$ then $y= f(x)$ where $ x\in X$, and since $y \in \bigcap_i Y_i$ necessarily  $x \in \bigcap_i f_e^{-1}[Y_i].$  So $f_e$ is onto on $F^\alpha$.
\end{proof}

\section{Quotient and Lift} \label{sec: quotient and lift}

For each connected subgraph $\alpha \in \mathcal{C}$ define
\begin{equation}
    \label{eq:W-def}
    W^\alpha = F^\alpha/ \sum_{\beta < \alpha} F^\beta.
\end{equation}

\begin{proposition}\label{pr:W-iso}
For each arrow $e$ in $\alpha$, if there is exactly one  virtual boundary arrow of $\Omega$ behind $e$, then  $f_e$ is bijective on $W^\alpha$.  For each vertex $i$ not in $\alpha$, $W^\alpha_i=\{0\}$. 
\end{proposition}

\begin{proof} Suppose $e$ is in $\alpha$ and there is exactly one  virtual boundary arrow of $\Omega$ behind it.  Then either that virtual boundary is a virtual boundary of $\alpha$ as well, or there is a boundary arrow of $\alpha$ behind $e$, so by Prop.~\ref{lm:V-onto}, $f_e$ is onto when restricted to $F^\alpha$, and similarly when restricted to $\sum_{\beta<\alpha} F^\beta$.  Thus it suffices to show that the kernel of $f_e$ in $F^\alpha$ is contained in $\sum_{\beta<\alpha} F^\beta$.  Let $\beta$ be the reduction of  $\alpha$ by the arrow $e$,  This is nonempty, and $F^\beta-{s(e)}= f_e^{-1}[0] \cap F^\alpha-{s(e)}$.  This implies the first sentence.

Suppose $i \in \alpha'_e$ for some boundary $e$.  If $\alpha$ is behind  $e$  then $F^\alpha_i=V^{e,-}_i=0$ and the second sentence follows.  If $\alpha$ is in front of $e$  then $F^\alpha_i$ is $F^\alpha_{t(e)}$ transported to $i$.  We will enhance $\alpha$ as follows.  Take $\beta$ to contain all the vertices of $\alpha$, and since there is a unique injective walk from $i$ to $t(e)$, add all the vertices in that path to $\beta$.  For any vertex you add to $\beta$, if it is the source of an arrow, add the target to $\beta$.  The resulting $\beta$ is an enhancement of $\alpha$ by $e$, and has the following properties:  it contains $i$, all the literal boundary arrows of $\beta$ which are not in $\alpha$ are inward, and the injective journey from $i$ to any of those boundary arrows, or to virtual boundary in $\beta$ but not $\alpha$, agrees in orientation with all its component arrows.  The intersection of all the representations associated to literal and virtual boundary that $\beta$ shares with $\alpha$ at the vertex $i$ will be $F^\alpha_i$.  The representation associated to a new boundary $e'$ or a new virtual boundary $E$ will be $V^{e',+}$ or $V^{E,+}$, and because there is a consistently oriented injective journey from $i$ to this boundary the value of that representation at $i$ is $V_i$.  Thus $F^{\beta}_i=F^\alpha_i$, and thus $\sum_{\beta < \alpha} F^\beta_i=F^\alpha_i$, and the quotient is trivial.
\end{proof}

In general we define an \emph{almost gradation} of a poset filtration $(F , P)$ of an $R$ module $M$ to be a function $C: P  \to \text{Sub}(M)$ satisfying the condition that for all $p \in P$, $F^p = F^{< p} \oplus C^p$, where we define $F^{<p} = \sum_{q < p} F^q$.

\begin{proposition} \label{pr:lifts}
For each connected subquiver $\alpha \in \mathcal{C}$, if every edge $e$ in $\alpha$ has exactly one virtual boundary of $\Omega$ behind it, the map $F^\alpha \to W^\alpha$ lifts, i.e. there exists an almost gradation of the poset filtration $F$. 
\end{proposition}

\begin{proof}  We need to lift each $W^\alpha_i$ to $F^\alpha_i$ consistently with the maps $f_e$.  Since all $F^\alpha_i$ outside of $\alpha$ are $0$, this is trivial except for vertices and arrows in $\alpha$.  So it suffices to choose lifts for each $i$ in $\alpha$, in such a way that for each arrow $e$ in $\alpha$ between $i$ and $j$ the following diagram commutes (using Lm.~\ref{lm:V-onto} and Prop.~\ref{pr:W-iso})
\begin{equation}
    \begin{tikzcd}
0 \arrow[r] & X^\alpha_j \arrow[r,hook] & F^\alpha_j \arrow[d,twoheadleftarrow,"f_e"] \arrow[r,two heads]  \arrow[r,bend right,dashleftarrow] & W^\alpha_j \arrow[d,equal,"f_e"] \arrow[r] &0\\
0 \arrow[r] & X^\alpha_i \arrow[r,hook] & F^\alpha_i \arrow[r,two heads]  \arrow[r,bend right,dashleftarrow] & W^\alpha_i \arrow[r] & 0
\end{tikzcd}
\end{equation}
where $X^\alpha=\sum_{\beta<\alpha} F^\beta$.  In the category of vector spaces we can lift one particular $W^\alpha_i$ to $F^\alpha_i$.  We will extend this lift recursively to each other vertex across each arrow as follows.  If the lift $l_i \maps W_i^\alpha \to F_i^\alpha$ is chosen  in the above diagram, then on $W_j^\alpha$ define the lift $f_e \circ l_i \circ f_e^{-1}$ (the inverse exists by Prop.~\ref{pr:W-iso}).  If on the other hand the the lift $l_j \maps W_j^\alpha \to F_j^\alpha$ is chosen  in the above diagram, the lift in $i$ is $f_e^{-1} \circ l_j \circ f_e$, where $f_e^{-1}$ is any lift of $f_e$.
\end{proof}

In general we say that an almost gradation $C: P \to \text{Sub}(M)$ of a poset filtration $F : P \to \text{Sub}(M)$ of an $R$ module $M$ \emph{spans} if $M = \bigoplus_{p \in P} C^p$.

\begin{proposition} \label{pr:spans}
The sum of the subrepresentations $F^\alpha$ spans $V$.
\end{proposition}

\begin{proof}  Let $i$ be a vertex.  Define $\alpha$ recursively by including $i$ in $\alpha$, and including the target of every arrow whose source is in $\alpha$.  $\alpha$ has the property that every literal boundary arrow  necessarily has $\alpha$ behind it  and for every vertex $j$, there is a path from $i$ to $j$.  Then $F^\alpha_i$ is the intersection of the transport of $V_j$ for various $j$ in $\alpha$, which is then $V_i$.  Thus $F^\alpha_i=V_i$. 
\end{proof}

\begin{proposition} \label{pr:multiplicity}
For each connected subquiver $\alpha \in \mathcal{C}$, if every edge $e$ in $\alpha$ has exactly one virtual boundary of $\Omega$ behind it, then $W_\alpha$ is an isotypic $\Omega$-representation.
\end{proposition}

\begin{proof}  Recall that $W$ is an isotypic $\Omega$-representation  if there is an indecomposable $\Omega$-representation $X$, a vector space $Y$, and for each vertex  $i$ an isomorphism $W_i \cong X_i \otimes Y$ so that each arrow $f^W_i$ is taken under the isomorphism to $f^X_i \otimes 1$.  If $\alpha$ satisfies the assumptions of the proposition then by Prop.~\ref{pr:W-iso} every $f_e$ is an isomorphism if $e$ is in $\alpha$ and $0$ otherwise.  Let $X^\alpha$ be the $\Omega$-representation which assigns the one-dimensional vector space $F$ to each vertex in $\alpha$ and the vector space $0$ to each vertex not in $\alpha$, the identity to each arrow in $\alpha$ and the trivial map to each arrow not in $\alpha$.  $X^\alpha$ is clearly indecomposable.  Choose one vertex $i$ in $\alpha$ and let $Y$ be the vector space $W^\alpha_i$.  Notice that for any other vertex $j$ of $\alpha$ there is a unique isomorphism $f$ between $W^\alpha_i$ and $W^\alpha_j$ consisting of a product of $f_e$'s and their inverses.  Define the isomorphism $X^\alpha_j \otimes Y = W^\alpha_i \to W^\alpha_j$ to be that isomorphism.  Check that this writes $W^\alpha$ as a representation with multiplicity.
\end{proof}

\section{Complete Decomposition of $A_{\infty, \infty}$.} \label{sec: complete decomp of type A}

Let $A_{\infty, \infty}$ be the graph with a vertex for each integer $i$ and an edge between each two adjacent integers, and let $\Omega$ be an eventually outward quiver with underlying graph $A_{\infty, \infty}$.  

\begin{figure}[h]
\begin{center}
\begin{tikzcd}
\ldots \arrow[dash, r, "a_{-3}"] & x_{-2} \arrow[dash, r, "a_{-2}"] &x_{-1} \arrow[dash, r, "a_{-1}"] &  x_0 \arrow[dash, r, "a_0"]  & x_1 \arrow[dash, r, "a_1"] &  x_2 \arrow[dash, r, "a_2"] & \ldots
\end{tikzcd}
\end{center}
\caption{The Graph $A_{\infty, \infty}$}
\label{fig:the A infinity infinity graph}
\end{figure}

Notice that every connected subquiver $\alpha$ of $\Omega$ has exactly two literal/virtual boundary arrows, a left boundary  which is either the literal left boundary arrow or the left  virtual  arrow of $\Omega$, and a right boundary which is either the literal right boundary arrow or the right virtual  arrow of $\Omega$.

We will define two total orderings, one $<_L$ on the set of all arrows of $A_{\infty, \infty}$ together with a symbol $-\infty$ (representing the left virtual boundary of $\Omega$), and the other $<_R$ on the set of all arrows in $A_{\infty, \infty}$ together with the symbol $\infty$ representing the right virtual arrow.  If $e$ and $e'$ are two arrows oriented the same way, then $e <_L e'$ and $e<_R e'$ if $e$ is behind $e'$.  If they are oriented oppositely then $e <_L - \infty <_L e'$ and $e' <_R \infty <_R e$ if $e$ is pointing to the left and $e'$ to the right. 
Note that the restriction of either $<_R$ or $<_L$ to the set of arrows $\Omega_1$ is a total order extension of the partial order $\prec$ from \ref{subsection: types of quivers and walks}, in the sense that for any pair of edges $e, e'$ if $e \prec e'$ then we have that $e <_R e'$ and $e <_L e'$. 

Recall that if $(P ,  \leq_P)$ and $(Q , \leq_Q)$ are two partially ordered sets, their \emph{product} $(P ,  \leq_P) \times (Q , \leq_Q)$ is the partially ordered set $(P \times Q , \leq_{P \times Q})$ where $(p , q) \leq_{P \times Q} (p' , q')$ if and only if $p \leq_P p'$ and $q \leq_Q q'$. 

\begin{lemma}\label{lem: order embedding of connected subquivers into product of edges}
There is an order-embedding $\iota$ (i.e. order -preserving, injective, and order-reflecting) of the poset $\mathcal{C}$ of connected subquivers of $\Omega$ with the reduction/enhancement partial order into the poset $(\Omega_1 \cup \{ -\infty\} , <_L) \times (\Omega_1 \cup \{ \infty \}, <_R )$ given by sending a connected subquiver $\alpha$ to the ordered pair of consisting of its left and right literal or virtual boundary arrows respectively. 
\end{lemma}

\begin{proof}
The map is clearly injective since a connected subgraph of $A_{\infty, \infty}$ is uniquely determined by its boundary arrows. To prove it is order-preserving, it suffices to show that if $\alpha, \beta \in \mathcal{C}$ are connected components and $\alpha > \beta$ with $\beta$ a reduction or enhancement of $\alpha$ then the left/right boundary arrow of $\alpha$ is smaller than the left/right boundary arrow of $\beta$. 

If $\beta$ is a reduction of $\alpha$ along an arrow $e$ in $\alpha$ then $\alpha$ and $\beta$ agree behind $e$, so each has a boundary arrow behind $e$ and these arrows are equal. The remaining boundary arrow $a$ of $\alpha$ is in front of $e$ which is the remaining boundary arrow of $\beta$. Regardless of the direction of $e$, we obtain that $a$ is strictly larger than $e$ in the $<_R$ or $<_L$ order, so the desired result follows. The case when $\beta$ is an enhancement of $\alpha$ is similar. 

To see that it is order-reflecting, assume $\iota(\alpha)>\iota(\beta)$ and show $\alpha>\beta$.  If $\iota(\alpha)=(i,j)$ and $\iota(\beta)=(k, \ell)$, then $k \leq i$ and $\ell \leq j$.  Let $\gamma$ be the unique connected subset such that $\iota(\gamma)=(k,j)$.  Then by the definition of $<_L$ $\gamma$ is an enhancement of $\alpha$ if $k$ is to the left of $i$ and an reduction if $k$ is to the right, in either case $\gamma< \alpha$ (or equal if $k=i$. Similarly we see by the definition of $<_R$ that $\beta<\gamma$.

\end{proof}

Associate to any arrow $e$ a subrepresentation of $V$ called $L^e$ which is $V^{e,+}$ if $e$ points to the right and $V^{e,-}$ if $e$ points to the left, and a subrepresentation of $V$ called $R^e$  which is $V^{e,+}$ if $e$ points to the left and $V^{e,-}$ if $e$ points to the right.   Associate to $-\infty$ a subrepresentation $L^{-\infty}=V^{E,+}$ where $E$ is the left virtual arrow of $A_{\infty, \infty}$  and associate to $\infty$ a subrepresentation $R^{\infty}=V^{E,+}$ where $E$ is the right virtual arrow of $A_{\infty, \infty}$. 

\begin{lemma}\label{lem: L intersect R is V}
$L$ is a poset filtration of $V$ with respect to the $<_L$ order and $R$ is a poset filtration of $V$ with respect to the $<_R$ order.  For each connected subquiver $\alpha$ the subrepresentation $V^\alpha$ is the intersection of $L^e$ and $R^{e'}$, where $e$ and $e'$ are respectively the left and right virtual or literal boundary arrows of $\alpha$.
\end{lemma}

\begin{proof}
We check the first sentence only for the $L$, the argument for $R$ is the same.  If $e$ is a literal arrow oriented to the right then $L^e = V^{e , +}$, and thus $L^e_i$ is $V_i$ to its left and a transport of $V_{t(e)}$ to its right. $L^{-\infty}_i$ is the transport of $V_j$ for $j$ sufficiently far to the left of it. If $e$ is oriented to the left, $L^e = V^{e , -}$, hence $L^e_i$ is $0$ to its left and a transport of $\{ 0 \}_{t(e)}$ to its right.  Thus if $e$ is oriented to the right, $L^e_i$ is bigger than any $L^{e'}_i$ to its left (behind it), including the left virtual  arrow, and if $e$ is oriented to the left  $L^e_i$ is smaller than any arrow to its left (in front of it), and smaller than $L^{-\infty}$.  

That $V^\alpha$ is the intersection of $L$ and $R$ is just a restatement of the definition of $V^\alpha$ in the case when there are always two boundary arrows.  
\end{proof}

Let $M$ be an $R$ module. We say that a poset filtration $(F ,P)$ of $M$ is \emph{distributive} if for all finite subsets $Q \subseteq P$ and all $p \in Q$ which is maximal in $Q$  we have $F_p \cap \sum_{q  \in Q \smallsetminus\{p\}} F_q \subseteq F_{<p}$. 

\begin{proposition}\label{prop: every almost gradation of a distributive poset filtration is independent}
Let $R$ be a ring and $M$ an $R$ module. If $(F , P)$ is a distributive poset filtration of $M$ then every almost gradation of $F$ is independent.
\end{proposition}

\begin{proof}
Let $C : P \to \text{Sub}(M)$ be an almost gradation of $F$, and suppose that $Q \subseteq P$ is a finite subset and for each $q \in Q$ we have $c_q \in C_q$ such that $\sum_{q \in Q} c_q = 0$. Since $Q$ is finite, there exists a maximal element $p \in Q$ of $Q$ (i.e. if $p \leq q$ for some $q \in Q$ then $p = q$). Then we write 
\begin{equation*}
c_p  =  - \sum_{q  \in Q \smallsetminus\{p\}} c_q.
\end{equation*}

Note that the LHS is contained in $F_p$, and the RHS is contained in $\sum_{q  \in Q-\{p\}} F_q$, so both sides are contained in $F_p \cap \sum_{q  \in Q \smallsetminus\{p\}} F_q \subset F_{<p}$ by distributivity. Thus $c_p \in F_{<p}$ and since $c_p \in C_p$ and $V_p=C_p \oplus F_{<p}$, it follows $c_p=0$, and then by induction on the size of $Q$, all $c_q=0$ as desired.
\end{proof}

\begin{proposition} \label{prop: intersection of linear filtrations are distributive}
  If $(E , I)$, and $(F , J)$ are linear filtrations of $M$ then for each $(i , j) \in I \times J$, we have that 
  \begin{equation}
    [E \cap F]_{< (i , j)} = E_i \cap F_{< j} + E_{< i} \cap F_j. \label{eq: intersection of linear less than}
  \end{equation} and therefore their intersection $(E \cap F , I \times J)$ is distributive. 
\end{proposition}

\begin{proof}
By definition, we have that $[E \cap F]_{< (i , j)} = \sum_{(k , \ell) < (i , j)} E_k \cap F_\ell$. We now show mutual inclusion of the desired equality. 

For the containment ($\subseteq$), suppose that $(k , \ell ) < (i , j)$, and thus either $k < i$ and $\ell \leq j$, or $k \leq i$ and $\ell < j$. Without loss of generality, suppose that $k < i$ and $\ell \leq j$. Then we have that $E_k \subseteq E_{<i}$, and $F_\ell \subseteq F_j$. Therefore $E_k \cap F_\ell \subseteq E_{< i} \cap F_j$.

To show the containment ($\supseteq$), we prove only that $E_i \cap F_{< j} \subseteq [E \cap F]_{< (i , j)}$, the other case being similar. Consider $v \in E_i \cap F_{< j} \subset \sum_{\ell < j} F_\ell $. in particular we have that $v \in F_{< j} = \sum_{\ell < j} F_\ell$. Since $(J , \leq)$ is totally ordered and the sum is manifestly nonempty, $\sum_{\ell < j} F_\ell = \bigcup_{\ell < j} F_\ell$. Therefore there exists some $\ell' < j$ such that $v \in F_{\ell '}$.  Hence $v \in E_i \cap F_{\ell '} \subseteq  [E \cap F]_{< (i , j)}$. 

Let $Q \subseteq I \times J$  and let $p=(i,j) \in Q$ be maximal.  If $x \in [E \cap F]_p \cap \sum_{ Q \smallsetminus\{p\}} [E \cap F]_q$, then  we can write $x = \sum_{(\ell , m) \in Q \smallsetminus \{p\}} x_{\ell , m}$ where $x_{\ell , m} \in E_\ell \cap F_m$. Because $p$ is maximal either  $\ell < i$ or $m < j$. Since each of the latter lie in $F_{<j}$ we have

\begin{equation*}
x - \sum_{\substack{(\ell , m) \in Q \\ \ell < i}} x_{\ell, m} \in F_{<j}.
\end{equation*}
But of course ever term in the left hand side is in $E_i$ so

\begin{equation*}
x - \sum_{\substack{(\ell , m) \in Q \\ \ell < i}} x_{\ell, m} \in E_i \cap F_{<j}.
\end{equation*}
On the other hand each $x_{\ell, m} \in E_{<i}$, so 
\begin{equation*}
x \in  E_{<i} + E_i \cap F_{<j}.
\end{equation*}
and noting that $x \in F_j$ and $E_i \cap F_{<j} \subset F_j$ gives
\begin{equation*}
x \in  E_{<i}  \cap F_j+ E_i \cap F_{<j} \subset [E \cap F]_{< (i , j)}
\end{equation*}
by Eq.~\eqref{eq: intersection of linear less than}
\end{proof}

A \emph{morphism of poset filtrations} $\varphi : (E , P) \to (F , Q)$ of an $R$ module $V$ is a order-preserving map $\varphi : P \to Q$ such that $F \circ \varphi = E$.  

\begin{lemma}\label{lem: pullback of distributivity}
Suppose that $\varphi : (E , P) \to (F , Q)$ is a morphism of poset filtrations which is an order-embedding, that $(F, Q)$ is distributive, and that and for all $p \in P$ we have that $F_{<\varphi(p)} \subseteq E_{<p}$. Then $(E , P)$ is distributive as well.  
\end{lemma}

\begin{proof}
Let $S \subseteq P$ be a finite subset and $p \in S$ be maximal. Then $\varphi(S) \subseteq Q$ is finite and $\varphi(p) \in \varphi(S)$ is maximal because $\varphi$ is order-reflecting. Then 
\begin{equation*}
E_p \cap \sum_{s \in S \smallsetminus\{p\}} E_s \overset{(1)}{=} F_{\varphi(p)} \cap \sum_{s \in S \smallsetminus \{p\}} F_{\varphi(s)} \overset{(2)}{=} F_{\varphi(p)} \cap \sum_{q \in \varphi(S) \smallsetminus \{\varphi(p)\}} F_{q} \overset{(3)}{\subseteq} F_{<\varphi(p)} \overset{(4)}{\subseteq} E_{<p}
\end{equation*}
where (1) follows because $\varphi$ is a poset filtration morphism, (2) follows because order-embeddings are injective, (3) follows because $(F , Q)$ is distributive, and (4) follows by hypothesis.
\end{proof}

\begin{lemma}\label{lem: connected subquivers to product satisfies distributive pullback condition}
The order-embedding $\iota : \mathcal{C} \to (\Omega_1 \cup \{ -\infty \} , <_L) \times (\Omega_1 \cup \{ \infty \} , <_R)$ defined in Lem. \ref{lem: order embedding of connected subquivers into product of edges} satisfies the condition that for all $\alpha \in \mathcal{C}$, $[L \cap R]_{ < \iota(\alpha) } \subseteq F_{<\alpha}$. 
\end{lemma}

\begin{proof}
Suppose that $\iota(\alpha)=(e_L,e_R)$, and by Equation \ref{prop: intersection of linear filtrations are distributive}, $[L \cap R]_{ < \iota(\alpha) } =L_{e_L} \cap R_{< e_R} + L_{< e_L} \cap R_{e_R}$. It suffices to show that $L_{e_L} \cap R_{< e_R}$ and $L_{< e_L} \cap R_{e_R}$ are contained in $F_{<\alpha}$, and we only show the former, the latter being similar. 

Note that $R_{< e_R} = \sum_{e <_R e_R} R_e = \bigcup_{e <_R e_R} R_e$, the last equality holding because $\Omega_1 \cup \{ \infty \}$ is linearly ordered and some $R_e$ is nonempty. Therefore $L_{e_L} \cap R_{< e_R} = L_{e_L} \cap \bigcup_{e <_R e_R} R_e = \bigcup_{e <_R e_R} L_{e_L} \cap R_e$. Define $\beta=\iota^{-1}(e_L,e)$ since $\iota$ is injective and note $\beta < \alpha$ since $\iota(\beta)<\iota(\alpha)$ and $\iota$ is order-reflecting. So $F_\beta \subset F_{<\alpha}$ and hence $L_{e_L} \cap R_{< e_R} \subset F_{<\alpha}$.
\end{proof}

\begin{corollary}\label{cor: every almost gradation of F is independent}
Let $V$ be an $A_{\infty, \infty}$-representation. Every almost gradation of the poset filtration $(F , C)$ is independent. 
\end{corollary}

\begin{proof}
By Lem. \ref{lem: order embedding of connected subquivers into product of edges} and Lem. \ref{lem: L intersect R is V} $\iota: (F , C) \to (L \cap R , \Omega_1 \cup \{-\infty \} \times \Omega_1 \cup \{\infty \})$ is an order-embedding morphism. By Prop. \ref{prop: intersection of linear filtrations are distributive} $(L \cap R , \Omega_1 \cup \{-\infty \} \times \Omega_1 \cup \{\infty \})$ is distributive. By Lem. \ref{lem: connected subquivers to product satisfies distributive pullback condition}, $\iota$ satisfies the necessary conditions to apply Lem. \ref{lem: pullback of distributivity}, which implies that $(F , C)$ is also distributive. Hence Prop. \ref{prop: every almost gradation of a distributive poset filtration is independent} gives the desired result. 
\end{proof}

\begin{theorem} \label{thm: main theorem}
If $V$ is an $A_{\infty, \infty}$-representation then $V$ can be written as a direct sum of representations, each of which is isomorphic to $W^\alpha$  for some connected subquiver $\alpha$, and each $W^\alpha$ is an isotypic. 
\end{theorem}

\begin{proof}
By Prop.~\ref{prop: F is a poset filtration}, $(F , C)$ is a poset filtration of subrepresentations of $V$. By Prop.~\ref{pr:lifts} $W^\alpha$ is an almost gradation of $(F , C)$, by Cor.~\ref{cor: every almost gradation of F is independent} $W^\alpha$ is independent, and by Prop:~\ref{pr:spans}, $W^\alpha$ spans $V$. By Prop.~\ref{pr:multiplicity} each $W^\alpha$ is an isotypic. 
\end{proof}

\section{Description of the Indecomposables} \label{sec: description of the indecomposables}

\begin{proposition} \label{pr: all indecomposables}
  For each full connected subquiver $\alpha \in \mathcal{C}$ the isomorphism class of representations $X^\alpha$ with $X^\alpha_i$ being zero dimensional if $i$ is not in $\alpha$ and $1$-dimensional if $i$ is in $\alpha$ and with $f_e$ a bijection for each arrow in $\alpha$ is indecomposable, and thus every representation of an eventually outward $A_{\infty, \infty}$ quiver is a direct sum of indecomposables $X^\alpha$.  
\end{proposition}

\begin{proof}
  Each $X^\alpha$  is clearly indecomposable, because if it were a sum of two subrepresentations then each is either one or zero dimensional at each vertex, and thus each is the sum of the $X^\alpha_i$ for some subset of the vertices and the other is the sum for the complementary set of vertices. If each is a nonempty set, since $\alpha$ is connected there must be a pair of vertices connected by an arrow from different subrepresentations.  Since the morphism associated to that edge is a bijection,  this is impossible. 

  By Props.~\ref{pr:W-iso} and~\ref{pr:multiplicity}, each $W^\alpha$ is a sum of copies of $X^\alpha$, and by Thm.~\ref{thm: main theorem} any representation is a sum of $X^\alpha$.
\end{proof}

Note that the indecomposables in Prop.~\ref{pr: all indecomposables} are exactly the indecomposable representations found by Bautista et. al.  \cite{bautista-liu-paquette2011} in their Prop.~5.9, although they deal only with locally finite representations.

There is a useful interpretation of this with the Euler form.  Let us define the \emph{root space} of the quiver to be the space of functions from $\Omega_0$ to the real numbers which are zero on all but finitely many vertices.  If $n$ and $m$ are two such functions define the Euler form to be
\begin{equation} \label{eq: tits form}
  \bracket{n,m}_\Omega= \sum_{ i \in \Omega_0} n_im_i - \frac{1}{2} \sum_{e \in \Omega_1}\parens{ n_{s(e)}m_{t(e)} + n_{t(e)}m_{s(e)}}
  \end{equation} 
where of course $n_i$ is the value of $n$ at $i$.

Observe that the Euler form is positive definitive because the corresponding Tits quadratic form satisfies
\[\bracket{n,n}_\Omega= \frac{1}{2}\sum_{i \in \Omega_0} (n_i+n_{i+1})^2 >0\]
unless every $n_i=-n_{i+1}$, which contradicts the finiteness.

Define the \emph{weight space} of the quiver to be the space of functions from $\Omega_0$ to the real numbers, with no restriction on the values. Observe that if $n$ is in the weight space and $m$ is in the root space then Eq.~\eqref{eq: tits form} still makes sense and thus the weight space can be thought of as functionals on the root space.  Define a uniform topology on the weight space with an entourage for each finite subset $S$ of the vertices, the open neighborhood of each $n$ being the weights that assign the same value as $n$ to every vertex in $S$. Equivalently, Cauchy sequences are those that for each such $S$ are eventually constant on each vertex of $S$.  It is easy to check that the weight space is the completion of the root space in this topology and that each root gives through the Euler form a continuous map from the weight space to the real numbers (the reals are given the discrete topology).

Observe that every representation of a quiver which is locally finite-dimensional (that is the vector space at each vertex is finite-dimensional) determines a natural number-valued weight called its dimension vector that assigns to each vertex the dimension of the associated vector space.

\begin{proposition} \label{pr: positive roots}
 For the quiver $A_{\infty, \infty}$ a  weight valued in the natural numbers is the dimension of an indecomposable representation if and only if it is in the closure of the set of all $\NN$-valued roots of length $1$.  
\end{proposition}

\begin{proof}
  By Prop.~\ref{pr: all indecomposables} the  representations $X^\alpha$ are the only indecomposables.  One checks that if $\alpha$  consists of finitely many vertices its dimension is of length $1$.  It is also clear that if $\alpha$ is infinite it is in the closure of the set of all finite connected subquivers $\alpha'$ contained in it.

  Suppose first $n$ is a root valued in the natural numbers, whose support $\alpha$ (vertices $i$ with $n_i>0$) is connected.  Then if $i$ and $j$ are its left and right endpoints
  \[\bracket{n,n}_\Omega= \frac{1}{2}(n_i^2+n_j^2) + \sum_{e \in \alpha} \frac{1}{2} \parens{n_{s(e)}-n_{t(e)}}^2.\]
  The first term is at least $1$ and only $1$ if $n+i=n_j=1$.  The second term is at least $0$ and only $0$ if all $n_k$ values in between are equal.  Therefore $\bracket{n,n}_\Omega=1$ if $n$ is all $1$s on $\alpha$ and $\bracket{n,n}_\Omega \geq 2$ otherwise.  If $n$ has disconnected support $\bracket{n,n}_\Omega$ is the sum of the values of each connected component, and therefore at least $2$.  So if $n$ is a root it is length $1$ only when it corresponds to $\alpha$.

  If $n$ is in the closure of the set of all $\NN$-valued roots of length $1$, it cannot assign a number bigger than $1$ to any vertex, because an open set around it based on a finite set of vertices containing that vertex would include no  $\NN$-valued roots of length $1$.  So it is the indicator function of a set of vertices.  If that set were not connected it has  a finite gap, and a finite set of vertices containing that gap and the points on either side would give an open set that includes no $\NN$-valued roots of length $1$. So every weight in the closure is of the desired form.

  The dimensions of course determine the indecomposable representation, because if a representation has the same dimensions as $X^\alpha$ but is not isomorphic, then for one edge $e \in \alpha$ we must have $f_e=0$.  In that case the representation decomposes as the sum of the subrepresentations supported respectively behind and in front of $e$.

\end{proof}

\section{A Representation which is not Krull-Schmidt} \label{sec: a non Krull-Schmidt representation}

In this section we give an example of a representation of a type $A_{\infty, \infty}$ quiver which is not Krull-Schmidt. It follows from Thm. \ref{thm: main theorem} that this quiver is not eventually outward. 

\begin{figure}[h]
\begin{center}
\begin{tikzcd}
\ldots \arrow[ r, "a_{-3}"] & x_{-2} \arrow[ r, "a_{-2}"] &x_{-1} \arrow[ r, "a_{-1}"] &  x_0  & x_1 \arrow[l, "a_0" '] &  x_2 \arrow[ l, "a_1" '] & \ldots \arrow[ l, "a_2" ']
\end{tikzcd}
\end{center}
\caption{The Quiver $\Omega$}
\label{fig:the omega quiver}
\end{figure}

Let $\Omega$ be the $A_{\infty, \infty}$ quiver in Figure \ref{fig:the omega quiver} with all arrows pointing towards the vertex $x_0$, and note that this quiver is not eventually outward. Let $V$ be the representation of $\Omega$ with $V_n = 0$ for all $n < 0$ and for $n \geq 0$, $V_n$ is the vector space of all sequences $(b_0, b_1, \ldots )$ with entries in $\mathbb{F}$ and $b_m = 0$ for all $m < n$. Let $f_n : V_{n + 1} \to V_n$ be the inclusion map for all $n \geq 0$. 

Suppose that $V$ is Krull-Schmidt, i.e. that $V = \bigoplus_{i \in I} V^i$ where for each $i \in I$, $V^i$ is an indecomposable representation of $\Omega$. 

\begin{lemma}
For all $k \geq 0$ there exists $i_k \in I$ and nonzero $x \in V^{i_k}_k$ such that $x \in V_k \smallsetminus f_k(V_{k + 1})$. 
\end{lemma}

\begin{proof}
Consider the standard basis vector $e_{k + 1} = (0 , \ldots, 0, 1 , 0 , \ldots) \in V_k \smallsetminus f_k(V_{k + 1})$. Since $V = \bigoplus_{i \in I} V^i$, we can write $e_{k + 1} = v_1 + \ldots + v_n$ where $v_m \in V^{i_m}_k$ and $i_m \in I$. It is not possible that $v_m \in f_k(V_{k + 1})$ for all $1 \leq m \leq n$ because then $e_{k + 1} \in f_k(V_{k + 1})$, which is false. Hence there must be some $v_m \in V_k \smallsetminus f_k(V_{k + 1})$, and taking $x = v_m$ and $i_k = i_m$ yields the desired result. 
\end{proof}

\begin{lemma}\label{lem: indecomposables are thin example}
If there exists $i \in I$ and nonzero $x \in V^i_k$ such that $x \in V_k \smallsetminus f_k(V_{k + 1})$ then $V^i_m = 0$ if $m < 0$ or $m > k$ and $V^i_m = \operatorname{Span}(x)$ for all $0 \leq m \leq k$. 
\end{lemma}

\begin{proof}
We prove this by induction on $k$. Suppose $k \geq 0$ and assume the result holds for all $0 \leq n < k$.

First prove that given $n < k$, $f_n: V^i_{ n + 1 } \to V^i_n$ is surjective. Given $v \in V^i_n$, if $v \in V_n \smallsetminus f_n(V_{n + 1})$, then by induction $V^i_m$ is $0$ for $m > n$ and in particular $V^i_k = 0$ contradicting that $x \neq 0$. Therefore it must be that $v \in f_n(V_{n + 1})$, and we call its preimage $v'$. Then by our direct sum decomposition of $V$ we can write $v' = \sum_{t = 1}^k v_{i_t}$ where $v_{i_t} \in V^{i_t}_{n + 1}$. But then $v = f_n(v') = \sum_{t = 1}^k f_n(v_{i_t})$ and $f_n(v_{i_t}) \in V^{i_t}_{n}$. Since $v \in V^i_n$ it must be that there is some $s$ such that $i_s = i$ and $v = f_n(v_{i_s})$ but $f_n(v_{i_t}) = 0$ for all $t \neq s$. Since $f_n$ is injective, it follows that $v' = v_{i_s} \in V^i_{n + 1}$ and $v_{i_t} = 0$ for all $t \neq s$, hence $v \in f_n(V^i_{ n + 1})$ as desired. 

Given $x \in V^i_k$ such that $x \in  V_k \smallsetminus f_k(V_{k + 1})$  let $C$ and $D$ be the subrepresentations of $\Omega$ given respectively below.
\begin{align*}
\ldots \to 0 \to &f\parens{\operatorname{Span}(x)} \overset{f_0}{\longleftarrow} \ldots \overset{f_{k - 1}}{\longleftarrow} \operatorname{Span}(x) \overset{f_k}{\longleftarrow} 0 \ldots
\\ \ldots \to 0 \to &ff_k(V_{k + 1}) \overset{f_0}{\longleftarrow} \ldots \overset{f_{k - 1}}{\longleftarrow} f_k(V_{k + 1}) \overset{f_k}{\longleftarrow} V_{k + 1} \overset{f_{k + 1}}{\longleftarrow} \ldots
\end{align*}
Here $f = f_0 \circ \ldots \circ f_{k-1}$. Note that $C$ is a  subrepresentation of $V^i$ and thus $C \oplus [ D \cap V^{i} ] \subset V_i$ (since $x \notin f_k(V_{k + 1})$, it follows that $C_k$ and $D_k \cap V^{i}_k$ are independent). In fact $V^{i} = C \oplus [ D \cap V^{i} ]$. For $n > k$ it is obvious that $V^{i}_n = C_n \oplus [ D \cap V^{i} ]_n$. If $v \in V^i_k \subset V_K$ then $v =\alpha x + w$ where $\alpha \in \mathbb{F}$ and $w \in f_{k+1}\parens{V_{k+1}}$, but in fact $w= v-\alpha x \in V^i$ and thus $w \in V^i_k \cap f_{k+1}\parens{V_{k+1}}$, and so $v \in C \oplus (D \cap V^i)$.  Thus $V^i_k= C_k \oplus (D_k \cap V^i_k)$.  Since $f_n$ is injective and surjective on $V^i$ for $0 \leq n <k$ the same is true for every $n \geq 0$, and both side are $0$ for $n<0$, so $V^i=C \oplus (D \cap V^i)$.

Since $V^i$ is indecomposable, this implies that $D \cap V^i$ is empty and $V^i=C$. 
\end{proof}

\begin{corollary}\label{cor: form of indecomposables}
All $V^i$ are of the form 
\begin{equation*}
\ldots \to 0 \to \operatorname{Span}(x) \overset{f_0}{\longleftarrow} \ldots \overset{f_{k - 1}}{\longleftarrow} \operatorname{Span}(x) \overset{f_k}{\longleftarrow} 0 \ldots
\end{equation*}
for some $k \geq 0$. 
\end{corollary}

\begin{proof}
Given $i \in I$, since $V^i$ is indecomposable, it is nonzero. Therefore there exists a nonzero vector $x$ in $V^i_\ell$ for some $\ell, \geq 0$. If $x \in V_\ell \smallsetminus f_\ell(V_{\ell+1})$ then by Lem. \ref{lem: indecomposables are thin example} we are done. If $x =f_\ell (y)$ for $y \in V_{\ell+1}$  write $y=\sum y_{t}$ for  $y_t \in V^{i_t}$ and thus $f_\ell(y_t) \in V^{i_t}$ and by independence of the direct summands one $i_t$ must equal $i$ and $y \in V^i$.  Repeating this process we must come to a $z \in V^i_k \smallsetminus f_k\parens{V_k}$,  since every nonzero sequence has a first nonzero entry. 
\end{proof}

\begin{lemma} \label{lem: at most one indecomposable per natural number}
For each $k \geq 0$ there is at most one $i \in I$ such that $V^i_k \cap [V_k \smallsetminus f_k(V_{k + 1})]$ is nonempty. 
\end{lemma}

\begin{proof}
Suppose that $i, j \in I$ are such that there exists $x \in V^i_k$, $y \in V^j_k$ and $x , y \in V_k \smallsetminus f_k(V_{k + 1})$. By Lem. \ref{lem: indecomposables are thin example}, we have that $V^i$ and $V^j$ are equal, respectively, to the following subrepresentations.  
\begin{align*}
\ldots \to 0 \to &\operatorname{Span}(x) \overset{f_0}{\longleftarrow} \ldots \overset{f_{k - 1}}{\longleftarrow} \operatorname{Span}(x) \overset{f_k}{\longleftarrow} 0 \ldots
\\ \ldots \to 0 \to &\operatorname{Span}(y) \overset{f_0}{\longleftarrow} \ldots \overset{f_{k - 1}}{\longleftarrow} \operatorname{Span}(y) \overset{f_k}{\longleftarrow} 0 \ldots
\end{align*}
Now by hypothesis, $V_{k + 1} = \bigoplus_{\ell \in I} V^\ell_{k + 1}$, and then because $f_k$ is injective, we have $f_k(V_{k + 1}) = \bigoplus_{\ell \in I} f_k(V^\ell_{k + 1})$. By Cor. \ref{cor: form of indecomposables}, either $V^\ell_{k + 1} = 0$ or $f_k$ is an isomorphism when restricted to $V^\ell_{k + 1}$, hence we can write $f_k(V_{k + 1}) = \bigoplus_{\ell \in L} V^\ell_{k + 1}$ with $L = \{ \ell \in I : V^\ell_{k + 1} \neq 0 \}$. 

Again by hypothesis, $V_k = \bigoplus_{\ell \in I} V^\ell_{k}$ and hence the family $(V^i_k, V^j_k, V^\ell_k : \ell \in L)$ of $V_k$ is independent. But this contradicts the fact that $V_k / V_{k + 1}$ is $1$-dimensional. 

\end{proof}

\begin{corollary}
$V$ is not Krull-Schmidt. 
\end{corollary}

\begin{proof}
If $V$ were a direct sum of indecomposable subrepresentations, say $V = \bigoplus_{i \in I} V^i$, then by Cor. \ref{cor: form of indecomposables}, each of the indecomposables is of the form described in said corollary, which means for all $i \in I$ there exists $k \geq 0$ such that $V^i_k \cap [V_k \smallsetminus f_k(V_{k + 1})]$ is nonempty. But by Lem. \ref{lem: at most one indecomposable per natural number} there is at most one such $i$ for each $k$, hence $I$ must be countable. However $V^i_0$ is one-dimensional for each $i$ and $V_0$ has uncountable dimension, yielding a contradiction. 
\end{proof}


\begin{thebibliography}{EER09}

\bibitem[BLP11]{bautista-liu-paquette2011}
Raymundo Bautista, Shiping Liu, and Charles Paquette.
\newblock Representation theory of an infinite quiver.
\newblock {\em arXiv preprint arXiv:1109.3176}, 2011.

\bibitem[Bri08]{brion2008representations}
Michel Brion.
\newblock Representations of quivers.
\newblock
  \url{https://www-fourier.ujf-grenoble.fr/~mbrion/notes_quivers_rev.pdf},
  2008.

\bibitem[EER09]{enochs2009injective}
E~Enochs, Sergio Estrada, and JR~Garc{\'\i}a Rozas.
\newblock Injective representations of infinite quivers. applications.
\newblock {\em Canadian Journal of Mathematics}, 61(2):315--335, 2009.

\bibitem[Gab72]{gabriel1972}
Peter Gabriel.
\newblock Unzerlegbare darstellungen i.
\newblock {\em Manuscripta mathematica}, 6(1):71--103, 1972.

\bibitem[Rin16]{ringel2016}
Claus~Michael Ringel.
\newblock Representation theory of {Dynkin} quivers. {Three} contributions.
\newblock {\em Frontiers of Mathematics in China}, 11(4):765--814, 2016.

\end{thebibliography}
\end{document}